\documentclass[12pt]{article}
\usepackage{amsmath,amssymb,color,amscd}

\def\seq#1#2#3{#1_{#2},\,\ldots,#1_{#3}}
\def\abs#1{\vert#1\vert}

\def\tt{{\underline{t}}}

\def\mm{\underline{m}}
\def\MM{\underline{M}}

\def\1{\underline{1}}

\def\Z{\mathbb Z}

\def\C{\mathbb C}

\def\OO{{\mathcal O}}

\def\DD{{\mathcal D}}

\def\oE{\stackrel{\circ}{E}}

\newtheorem{theorem}{Theorem}

\newtheorem{proposition}{Proposition}

\newenvironment{proof}
{\noindent{\bf Proof\/}.}{{ $\Box$}\smallskip\par}

\title{On the topological type of a set of plane valuations with symmetries
\footnote{Math. Subject Class. 14B05, 13A18,
14R20.
Keywords: finite group actions, Poincar\'e series,  plane valuations, topological
type.}
}

\author{
A.~Campillo,
\and F.~Delgado,\thanks{The first two authors were supported by the grant
MTM2015-65764-C3-1-P
(with the help of FEDER Program).} \and S.M.~Gusein-Zade
\thanks{
The work of the third author (Sections~1 and 2) was supported by the grant
16-11-10018 of the Russian Science Foundation
} }

\date{}
\begin{document}
\def\eps{\varepsilon}

\maketitle

\begin{abstract}
Let $\{C_i : i=1,\ldots,r\}$ be a set of irreducible plane curve singularities. For
an action of a finite group $G$, let $\Delta^{L}(\{t_{a i}\})$ be the Alexander
polynomial in $r\abs{G}$ variables of the algebraic link
$( \bigcup\limits_{i=1}^{r}\bigcup\limits_{a\in G}a C_i )\cap
S^3_{\varepsilon}$
and let $\zeta(\seq t1r) = \Delta^{L}(\seq t11, \seq t22, \ldots, \seq trr)$ with
$\abs{G}$ identical variables in each group. (If $r=1$, $\zeta(t)$ is the monodromy
zeta function of the function germ $\prod\limits_{a\in G} a^*f$, where $f=0$ is
an equation defining the curve $C_1$.)
We prove that $\zeta(\seq t1r)$ determines the topological type of the link $L$. We
prove an analogous statement for plane divisorial valuations formulated in terms of
the Poincar\'e series of a set of valuations.
\end{abstract}

\section{Introduction}\label{sec1}
An equivariant (with respect to an action of a finite group $G$) version of the Poincar\'e series
of a multi-index filtration (defined by a collection of valuations $\{\nu_i\}$, $i=1, \ldots, r$)
was defined in \cite{RMC-2015} as an element $P^G_{\{\nu_i\}}(t_1, \ldots, t_r)$ of the ring
$\widetilde{A}(G)[[t_1, \ldots, t_r]]$ of power series with coefficients in a certain modification
$\widetilde{A}(G)$ of the Burnside ring $A(G)$ of the group $G$. In \cite{DM} it was shown that,
for a filtration on the ring $\OO_{\C^2,0}$ of germs of functions in two variables defined either
by a collection of divisorial valuations or by a collection of curve valuations (in the latter case
with certain exceptions), the series $P^G_{\{\nu_i\}}(t_1, \ldots, t_r)$ determines the (weak)
equivariant topological type of the set of valuations.
(In the case of curve valuations this means the equivariant embedded topological type of the curve or
of the corresponding algebraic link.) (In \cite{FAOM} it was shown that even in the non-equivariant
case (i.~e., for $G=\{e\}$) the Poincar\'e series of a collection including both divisorial and curve
valuations does not determine, in general, the topological type of the set of valuations.)

One has a natural homomorphism $\varphi$ from the ring $\widetilde{A}(G)$ to the ring $\Z$ of integers
which sends a finite $G$-set (with an additional structure) to its number of elements. Applying the
homomorphism $\varphi$ to the equivariant Poincar\'e series $P^G_{\{\nu_i\}}(t_1, \ldots, t_r)$
(that is to the coefficients of it) one gets the series $\zeta(t_1, \ldots, t_r)=
P(t_1, \ldots, t_1, t_2, \ldots, t_2, \ldots, t_r, \ldots, t_r)$ (with $\vert G\vert$ identical variables
in each group), where $P(\{t_{a,i}\})$ is the usual Poincar\'e series (in $r\vert
G\vert$ variables)
of the set of valuations $\{a^*\nu_i\}$, $a\in G$, $i=1, \ldots, r$.

If $\nu_i$, $i=1, \ldots, r$, are the curve valuations corresponding to irreducible plane curve
singularities $(C_i,0)\subset(\C^2,0)$, then $P(\{t_{a,i}\})$ coincides with the
Alexander polynomial
$\Delta^{L}(\{t_{a,i}\})$ of the algebraic link
$L=\left(\bigcup\limits_{i=1}^r\bigcup\limits_{a\in G} aC_i\right)\cap
S^3_{\varepsilon}$,
where $S^3_{\varepsilon}$ is the sphere of small radius $\varepsilon$ centred at the origin in $\C^2$: \cite{Duke}.
Strictly speaking this holds if all the curves $aC_i$ are different.
If, in a collection
of plane curve singularities $(Y_i,0)\subset(\C^2,0)$, $i=1, \ldots, s$, two curves coincide (say,
$Y_{s-1}$ and $Y_s$ and only they), the Alexander polynomial of the corresponding link
$\left(\bigcup_{i=1}^s Y_i\right)\cap S^3_{\varepsilon}$ should be defined as
$\Delta^{L'}(t_1, \ldots, t_{s-2}, t_{s-1}\cdot t_{s})$, where
$\Delta^{L'}(t_1, \ldots, t_{s-2}, t_{s-1})$ is the usual Alexander polynomial of the link
$L'=\left(\bigcup_{i=1}^{s-1} Y_i\right)\cap S^3_{\varepsilon}$ with $(s-1)$
components.
If $r=1$ and the curve $C_1$ is defined by an equation $f_1=0$ ($f_1\in \OO_{\C^2,0}$), the
series $P(t,\ldots, t)$ coincides with the monodromy zeta function of the germ
$\prod\limits_{a\in G}a^*f_1$ (see, e.\ g., \cite{AGV}).

For a collection $\{\nu_i\}$, $i=1,\ldots, r$, consisting of divisorial and curve valuations on
$\OO_{\C^2,0}$ one has the following A'Campo type formula for the Poincar\'e series: \cite{Duke, DG}.
Let $\pi:(X,\DD)\to(\C^2,0)$, $\DD=\pi^{-1}(0)$, be a resolution of the collection $\{\nu_i\}$
of valuations. This means that $\pi$ is a modification of the plane (by a finite number of
blow-ups of points) such that all the divisors defining the divisorial valuations from the collection
are present in the exceptional divisor $\DD$ and the strict transforms of all the curves defining the
curve valuations do not intersect each other in $X$ and are transversal to $\DD$ (at
its smooth points).
All the components $E_{\sigma}$ of the exceptional divisor $\DD$ are isomorphic to the complex
projective line. Let $\oE_{\sigma}$ be ``the smooth part of $E_{\sigma}$ in the resolution'', that is
$E_{\sigma}$ itself minus the intersection points with other components of the exceptional divisor
$\DD$ and with the strict transforms of the curves defining the curve valuation. A {\it curvette}
at the component $E_{\sigma}$ is the image in $(\C^2,0)$ of a smooth curve germ transversal to
$\oE_{\sigma}$ at a point of it. Let $\varphi_{\sigma}=0$ be an equation of a curvette at $E_{\sigma}$
and let $m_{\sigma i}:=\nu_i(\varphi_{\sigma})$,
$\mm_{\sigma}:=(m_{\sigma 1},\ldots,m_{\sigma r})\in \Z_{\ge 0}^r$.
Then one has
\begin{equation}\label{ACampo}
P_{\{\nu_i\}}(\tt)=\prod_{\sigma}\left(1-\tt^{\mm_{\sigma}}\right)^{-\chi(\oE_{\sigma})},
\end{equation}
where $\tt=(t_1,\ldots, t_r)$, $\tt^{\mm_{\sigma}}=t_1^{m_{\sigma 1}}\cdots t_r^{m_{\sigma r}}$,
$\chi(\cdot)$ is the Euler characteristic.

A formula for the Alexander polynomial $\Delta^L(t_1,\ldots, t_r)$ in several variables
of an algebraic link $L$ in terms of a resolution of the corresponding curve can be found in \cite{EN}.
If all the valuations in the collection $\{\nu_i\}$ are curve ones, the equation~(\ref{ACampo}) and
the formula from \cite{EN} for the corresponding algebraic link $L$ give the same results, i.~e.,
$$
P_{\{\nu_i\}}(t_1,\ldots, t_r)=\Delta^L(t_1,\ldots, t_r).
$$

Here we show, in particular, that the ``usual'' (not equivariant) topological type of the curve
$\bigcup\limits_{i=1}^r\bigcup\limits_{a\in G} aC_i$ is determined by the series
$\zeta(t_1, \ldots, t_r)$
(that is by the described reduction of the Alexander polynomial in $r\vert G\vert$ variables).
We prove an analogous statement for a collection of plane divisorial valuations (with certain
precisely described exceptions).

One can get the impression that the results here are somewhat weaker than those in \cite{DM}
because from formal point of view they describe the usual (not equivariant) topology of the set
of valuations (of the curve if all the valuations are curve ones). However, the difference here
is not too big. In the setting of \cite{DM}, i.e., if the action of the group is induced by its
action (a representation) on $(\C^2,0)$, the series $\zeta(\ldots)$ permits to restore the (minimal)
equivariant resolution graph of the set of valuations. The only object which is missed is the
representation of the group on $\C^2$. This cannot be read from the series $\zeta(\ldots)$ and for that
in \cite{DM} one used the equivariant Poincar\'e series. (Even in that case this was possible not
always, but with certain exceptions.) Thus if the representation on $\C^2$ is given in advance,
the outputs of the results of \cite{DM} and those here are essentially the same. On the other hand here
the information is extracted from a considerably ``smaller'' invariant: a series with coefficients
in $\Z$, not in the Burnside ring of the group. We believe that this makes the results considerably
stronger. (This also makes the proofs somewhat more complicated.)

\medskip

It is well known that the Alexander polynomial determines the topological type of an algebraic knot.
Moreover, the topological type of an algebraic link $L=C\cap S^3_{\varepsilon}$ with $r$ components
($(C,0)\subset(\C^2,0)$ is a plane curve singularity)
is determined by its Alexander polynomial $\Delta^L(t_1, \ldots, t_r)$ in several variables (the number
of variables being the number $r$ of components of the link): \cite{Yamamoto}. On the other hand it is
known that the Alexander polynomial in one variable (that is $\Delta^L(t, \ldots, t)$) does not determine
the topological type of an algebraic link with at least two components (see, e.~g., an example in
\cite{Yamamoto}).

The statement above says, in particular, that, if the curve $(C,0)\subset(\C^2,0)$ defining the link $L$ consists of the
$a$-shifts of an irreducible plane curve singularity for all $a\in G$, its Alexander
polynomial in one
variable determines the topological type of the curve (or of the link). In this case all the components
of the curve $C$ are equisingular, that is have the same topological type. The attempt to understand
whether it is really necessary to have a symmetry (defined by a group) between the branches of the curve,
or it is sufficient that all the components of the curve are equisingular, led to the example
of two algebraic links with unknotted components and with equal Alexander polynomials
in one variable.
In terms of the Singularity Theory this example can be interpreted in the following way. The link
corresponding to a curve $\{f=0\}$ consists of unknotted components if and only if $f$ is the product of
function germs without critical points (that is of germs right equivalent to a coordinate function).
In this way the example gives two functions of this sort with equal monodromy zeta functions.

\section{Topology of divisorial valuations}\label{sec2}

In \cite{DM} we considered collections of curve (or divisorial) plane valuations consisting of
the orbits of some of them under an action of a finite group $G$ on the plane $(\C^2,0)$.
Here we consider a slightly more general setting
which can be applied to some other situations.

Let $\pi:(X,\DD)\to(\C^2,0)$, $\DD=\pi^{-1}(0)$, be a modification of the plane by a finite number of
blow-ups of points. All the components $E_{\sigma}$ of the exceptional divisor $\DD$ are isomorphic to
the complex projective line. The (dual) graph $\Gamma$ of the modification is defined in the following way.
Its vertices are in a one-to-one correspondence with the components $E_{\sigma}$ of the exceptional
divisor $\DD$. Two vertices are connected by an edge if and only if the corresponding components intersect.
The graph $\Gamma$ is a tree.
Each vertex $\sigma$ of $\Gamma$ (that is a component $E_{\sigma}$ of the exceptional divisor $\DD$) has its age:
the minimal number of blow-ups needed to create the component. The dual graph of a modification
with the ages of the vertices determines the combinatorics of the modification.

A divisorial valuation on the ring $\OO_{\C^2,0}$ of germs of functions in
two variables is defined by a component of the exceptional divisor of a modification. This modification
is called a {\it resolution\/} of the (divisorial) valuation. A modification which is a resolution of each
divisorial valuation from a (finite) collection $\{\nu_i\}$ is called a {\it resolution of
the collection\/}.

Two collections $\{\nu_i\}$ and $\{\nu'_i\}$ of divisorial valuations are called {\it topologically
equivalent\/} if they have isomorphic minimal resolution graphs $\Gamma$ and $\Gamma'$, i.~e.,  if there exists an isomorphism between the abstract graphs $\Gamma$ and $\Gamma'$
preserving the ages
and sending the vertices corresponding to the valuations $\nu_i$ to the vertices corresponding to
the valuations $\nu'_i$.

Assume that the graph $\Gamma$ of a modification carries an action of a finite group $G$ preserving the
ages of the vertices. Let $\nu_i$, $i=1, \ldots, r$, be divisorial valuations corresponding to some
vertices of $\Gamma$, and let $\nu_{ai}:=a^*\nu_i$, $a\in G$, be the divisorial
valuation defined by the
$a$-shift of the corresponding vertex. An important example of this situation
(treated in \cite{DM})
is the case when the group $G$ acts (analytically) on $(\C^2,0)$ and the modification
$\pi$ is
$G$-invariant. In fact, in the constructions below, the structure of the group $G$ is not really
important. We use
only the order $h_0$ of the group $G$ (this order is assumed to be known) and the orders of
its subgroups.

Let $\check{\Gamma}$ be the quotient $\Gamma/G$ of the modification graph $\Gamma$ by the group action.
It is a graph of a modification. (One can say that the modification $\pi$ above is ``an equivariant
extension'' of this one.) To avoid some difficulties (and/or ambiguities) in the descriptions and in the
notations below, we shall usually assume that the graph $\check{\Gamma}=\Gamma/G$ is embedded into
the graph $\Gamma$ (as a ``section'' of the quotient map). This can be made in many ways, but we
shall assume that one embedding is fixed. This permits to assume that all the vertices corresponding
to the valuations $\nu_i$ lie in $\check{\Gamma}$.
As above, for a vertex $\delta\in \Gamma$, let $\varphi_{\delta}=0$,
$\varphi_{\delta}\in \OO_{\C^2,0}$, be an equation of a curvette at the component $E_{\delta}$,
$m_{\delta i}:=\nu_i(\varphi_{\delta})$.
Let $M_{\delta i} := \sum_{a\in G} m_{(a\delta) i} = \sum_{a\in G}(a^*
\nu_i)(\varphi_{\delta})$
and $\underline{M}_{\delta}= (M_{\delta 1}, \ldots, M_{\delta r})\in \Z_{\ge 0}^r$.
The ``multiplicities" $\underline{M}_{\delta}$ are the same for the vertices from one
orbit. Therefore they depend only on the corresponding vertex in
the quotient graph $\check \Gamma$.
All the multiplicities $\MM_\sigma$, $\sigma\in \check \Gamma$, are different and for
$\sigma,
\tau\in \Gamma$ one has
$\MM_\sigma=\MM_\tau$ if and only if $\tau = a \sigma$ for some $a\in G$.

Let $P(\{t_{ai}\})$ be the Poincar\'e series of the collection $\{\nu_{ai}\}$ of $r\vert G\vert$
valuations and let
$$
\zeta(t_1,\ldots, t_r):=P(t_1,\ldots, t_1,t_2,\ldots, t_2, \ldots, t_r,\ldots, t_r)
$$
with $\vert G\vert$ identical variables in each group.
Assume that either the number of edges at the initial vertex of the (minimal) modification graph
(that is of the only vertex with the age 1) is different from 2, or it is equal to 2, but these two
edges are not interchanged by the group action.

Let us show an example when this condition is not satisfied and one cannot determine the topological
type of a set of divisorial valuations from the corresponding series $\zeta(\cdot)$.
Let us consider
two modification graphs shown on Figure~\ref{fig1} with the obvious (non-trivial)
actions of the group
$\Z_2$ with 2 elements. The numbers at the vertices are the ages, the divisors under consideration
are marked by the circles.
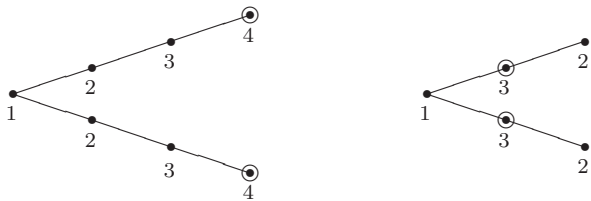
\begin{figure}[h]
$$
\unitlength=0.50mm
\begin{picture}(120.00,60.00)(0,0)
\thinlines
\put(-30,30){\line(3,1){63}}
\put(-30,30){\line(3,-1){63}}
\put(-30,30){\circle*{2}}
\put(-32,23){{\scriptsize$1$}}
\put(-9,37){\circle*{2}}
\put(-11,30){{\scriptsize$2$}}
\put(12,44){\circle*{2}}
\put(10,37){{\scriptsize$3$}}
\put(33,51){\circle*{2}}
\put(33,51){\circle{4}}
\put(31,44){{\scriptsize$4$}}
\put(-9,23){\circle*{2}}
\put(-11,16){{\scriptsize$2$}}
\put(12,16){\circle*{2}}
\put(10,8){{\scriptsize$3$}}
\put(33,9){\circle*{2}}
\put(33,9){\circle{4}}
\put(31,2){{\scriptsize$4$}}

\put(80,30){\line(3,1){42}}
\put(80,30){\line(3,-1){42}}
\put(80,30){\circle*{2}}
\put(78,23){{\scriptsize$1$}}
\put(101,37){\circle*{2}}
\put(101,37){\circle{4}}
\put(99,30){{\scriptsize$3$}}
\put(122,44){\circle*{2}}
\put(120,37){{\scriptsize$2$}}
\put(101,23){\circle*{2}}
\put(101,23){\circle{4}}
\put(99,16){{\scriptsize$3$}}
\put(122,16){\circle*{2}}
\put(120,9){{\scriptsize$2$}}
\end{picture}
$$
\caption{The modification graphs defining the divisorial valuations.}
\label{fig1}
\end{figure}
In the both cases one has $\zeta(t)=(1-t^5)^{-2}$ (see equation~(\ref{zeta}) below). Thus in these
cases one cannot determine the topological type of the set of valuations from the series $\zeta(t)$.

The main feature of this example is the fact that, for the initial vertex $\sigma_0$ of the modification graph (the only vertex with the age equal to 1), the Euler characteristic $\chi(\oE_{\sigma_0})$ is equal to zero. Therefore, in the A'Campo type formula, the binomial $(1-t^{M_{\sigma_0}})$ is absent (it is with the zero exponent) and one cannot determine the multiplicity $M_{\sigma_0}$ from the series $\zeta(t)$.
This problem does not appear in the considerations in \cite{DM} since the equivariant Euler characteristic of $\oE_{\sigma_0}$ (with values in the Burnside ring of the group $G$) is equal to
$2 [G/G]-[G/H]$ where $H$ is a subgroup of $G$ of index 2.

\begin{theorem}\label{theo1}
 In the described situation the series $\zeta(t_1,\ldots, t_r)$ determines the topological
 type of the collection $\{\nu_{ai}\}$ of divisorial valuations.
\end{theorem}

\begin{proof}
 Let $\nu=\nu_i$ be one of the valuations under consideration. Without loss of generality we can assume
 that $i=1$. The series $\zeta_{\nu}(t)$ corresponding to this valuation is determined from
 $\zeta(t_1,\ldots, t_r)$ by the following ``projection formula'':
 $$
 \zeta_{\nu}(t)=\zeta(t,1,\ldots, 1).
 $$

Let us show that, under the described conditions, one can restore the
minimal resolution
graph of the collection $\{a^*\nu\}$, $a\in G$. We shall do this using essentially
the series
$\zeta_{\nu}(t)$.
However, in a certain situation we shall look back at $\zeta(t_1,\ldots, t_r)$. The dual graph
$\Gamma$ of the minimal resolution of the divisorial valuations $\{a^*\nu\}$ is shown
in Figure~\ref{fig2}.
The quotient $\Gamma/G$ of this graph by the action of the group $G$ is the minimal resolution graph
$\check{\Gamma}$ of the valuation $\nu$ shown in Figure~\ref{fig3}.
Here $\sigma_q$, $q=0,1,\ldots,g$, are called the {\it dead ends\/} of the graph, $\tau_q$,
$q=1,\ldots,g$, are called the {\it rupture points\/} and $g$ is the number of the Puiseux pairs
of a curvette corresponding to the divisor defining $\nu$.
The graph $\Gamma$ can be obtained
from the graph $\check{\Gamma}$ by the following construction. There are several vertices $\rho_1$, \dots,
$\rho_{\ell}$ in the graph $\check{\Gamma}$ lying on the geodesic from $\sigma_0$ (the only vertex with
the age 1) to $\nu$, $\rho_1<\rho_2< \ldots< \rho_{\ell}$, and some numbers
$h_0=\vert G\vert>h_1>\ldots>h_{\ell}$ such that $h_{i+1}\vert h_i$.
The vertices $\rho_j$ (we shall call them the {\em splitting points}) are the images under the
quotient map of the points in $\Gamma$ in whose neighbourhoods the quotient map is not an isomorphism.
The number $h_j$ is the order of the isotropy subgroup for the vertices inbetween $\rho_{j-1}$ and
$\rho_j$ ($\rho_{j-1}$ excluded and $\rho_j$ included).
(Not all sequences $\rho_1<\rho_2< \ldots< \rho_{\ell}$ are permitted.) To get the graph $\Gamma$ from
the graph $\check{\Gamma}$ one takes $\vert G\vert$ copies of the latter one. The parts of all of them
preceding $\rho_1$ ($\rho_1$ included) are identified. The remaining parts preceding $\rho_2$
($\rho_2$ included) are identified in groups containing $h_0/h_1$ copies each. The remaining parts
preceding $\rho_3$ ($\rho_3$ included) are identified in groups containing $h_0/h_2$ copies each, etc.
Notice that $h_0/h_{\ell}$ is the cardinality of the orbit of $\nu$.

\begin{figure}[h]
$$
\unitlength=0.50mm
\begin{picture}(120.00,110.00)(0,-30)
\thinlines
\put(-10,30){\line(1,0){70}}
\put(60,30){\circle*{2}}

\put(60,30){\line(1,1){40}}
\put(85,55){\line(1,-1){10}}
\put(85,55){\circle*{2}}
\put(100,70){\circle*{2}}

\put(100,70){\line(2,1){20}}
\put(120,80){\circle{4}}
\put(120,80){\circle*{2}}
\put(110,75){\line(1,-2){4}}
\put(100,70){\line(2,-1){20}}
\put(120,60){\circle{4}}
\put(120,60){\circle*{2}}
\put(110,65){\line(-1,-2){4}}

\put(60,30){\line(1,-1){40}}
\put(100,-10){\circle*{2}}
\put(85,5){\line(-1,-1){10}}
\put(85,5){\circle*{2}}
\put(86,6){{\scriptsize$\tau_4$}}
\put(75,-5){\circle*{2}}
\put(73,-13){{\scriptsize$\sigma_4$}}

\put(98,-5){{\scriptsize$\rho_2$}}
\put(100,-10){\line(2,1){20}}
\put(120,0){\circle{4}}
\put(120,0){\circle*{2}}
\put(110,-5){\line(1,-2){4}}
\put(100,-10){\line(2,-1){20}}
\put(120,-20){\circle{4}}
\put(120,-20){\circle*{2}}
\put(125,-21){{\scriptsize$\nu$}}
\put(110,-15){\line(-1,-2){4}}
\put(102,-17){{\scriptsize$\tau_5$}}
\put(110,-15){\circle*{2}}
\put(106,-23){\circle*{2}}
\put(104,-31){{\scriptsize$\sigma_5$}}

\put(60,30){\line(1,0){40}}
\put(85,30){\line(0,-1){13}}
\put(85,30){\circle*{2}}
\put(100,30){\circle*{2}}

\put(100,30){\line(2,1){20}}
\put(120,40){\circle{4}}
\put(120,40){\circle*{2}}
\put(110,35){\line(1,-2){4}}
\put(100,30){\line(2,-1){20}}
\put(120,20){\circle{4}}
\put(120,20){\circle*{2}}
\put(110,25){\line(-1,-2){4}}

\put(-10,30){\circle*{2}}
\put(5,30){\line(0,-1){15}}
\put(5,30){\circle*{2}}
\put(5,15){\circle*{2}}
\put(25,30){\line(0,-1){20}}
\put(25,30){\circle*{2}}
\put(25,10){\circle*{2}}
\put(60,30){\line(0,-1){15}}
\put(60,15){\circle*{2}}

\put(-25,24){{\scriptsize ${\bf 1}=\sigma_0$}}
\put(6.5,10){{\scriptsize$\sigma_1$}}
\put(26.5,7){{\scriptsize$\sigma_2$}}
\put(58,10){{\scriptsize$\sigma_3$}}
\put(4,33){{\scriptsize$\tau_1$}}
\put(24,33){{\scriptsize$\tau_2$}}
\put(40,33){\scriptsize{$\tau_3=\rho_1$}}
\end{picture}
$$
\caption{The resolution graph $\Gamma$ of the valuations $\{a^*\nu\}$.}
\label{fig2}
\end{figure}
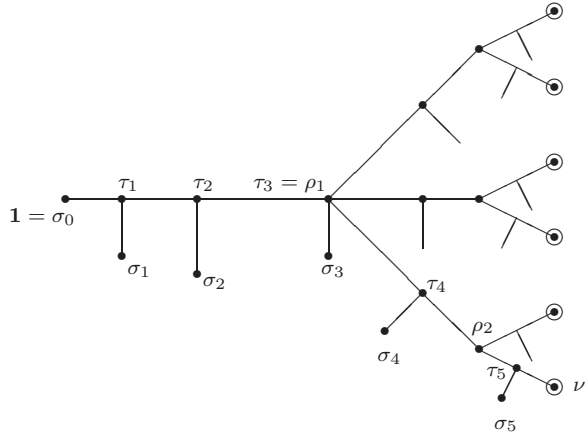

\begin{figure}
$$
\unitlength=0.50mm
\begin{picture}(120.00,40.00)(0,10)
\thinlines
\put(-10,30){\line(1,0){70}}
\put(60,30){\circle*{2}}

\put(82,33){{\scriptsize$\tau_4$}}
\put(82,12){{\scriptsize$\sigma_4$}}

\put(98,33){{\scriptsize$\rho_2$}}
\put(124,29){{\scriptsize$\nu$}}
\put(107,33){{\scriptsize$\tau_5$}}
\put(107,16){{\scriptsize$\sigma_5$}}

\put(60,30){\line(1,0){60}}
\put(85,30){\line(0,-1){13}}
\put(85,30){\circle*{2}}
\put(85,17){\circle*{2}}
\put(100,30){\circle*{2}}

\put(120,30){\circle{4}}
\put(120,30){\circle*{2}}
\put(110,30){\line(0,-1){8}}
\put(110,30){\circle*{2}}
\put(110,22){\circle*{2}}

\put(-10,30){\circle*{2}}
\put(5,30){\line(0,-1){15}}
\put(5,30){\circle*{2}}
\put(5,15){\circle*{2}}
\put(25,30){\line(0,-1){20}}
\put(25,30){\circle*{2}}
\put(25,10){\circle*{2}}
\put(60,30){\line(0,-1){15}}
\put(60,15){\circle*{2}}

\put(-25,24){{\scriptsize ${\bf 1}=\sigma_0$}}
\put(6.5,10){{\scriptsize$\sigma_1$}}
\put(26.5,7){{\scriptsize$\sigma_2$}}
\put(58,10){{\scriptsize$\sigma_3$}}
\put(4,33){{\scriptsize$\tau_1$}}
\put(24,33){{\scriptsize$\tau_2$}}
\put(50,33){\scriptsize{$\tau_3=\rho_1$}}
\end{picture}
$$
\caption{The resolution graph $\check \Gamma = \Gamma/G$ of the valuation
$\nu$.}
\label{fig3}
\end{figure}
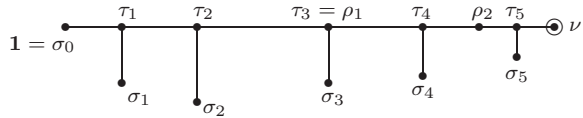

As above, we shall assume
that the graph $\check{\Gamma}=\Gamma/G$ is embedded into the graph $\Gamma$ (as a ``section'' of
the quotient map). (The number of embeddings is equal to $h_0/h_{\ell}$, we shall assume that one
embedding is fixed.) This embedding is shown in Figure~\ref{fig2} by indicating the
same names for some
vertices in $\Gamma$ as in $\check{\Gamma}$.

Also as above, for a vertex $\delta\in \check\Gamma$, let $\varphi_{\delta}=0$,
$\varphi_{\delta}\in \OO_{\C^2,0}$, be an equation of a curvette at the component $E_{\delta}$,
$m_{\delta}=\nu(\varphi_{\delta})$. (We use the same notations $m_{\bullet}$ as for the corresponding
multiplicities in the graph $\Gamma$, since for the vertices in the image of the graph $\check{\Gamma}$
they coincide.) It is known that the numbers $m_{\sigma_0}$, $m_{\sigma_1}$,\dots,
$m_{\sigma_g}$ form the minimal set of generators of the semigroup of values of the valuation $\nu$,
$m_{\tau_i}$ is a multiple of $m_{\sigma_i}$, $i=1, \ldots, g$: $m_{\tau_i}=N_i m_{\sigma_i}$,
and $m_{\sigma_0}=N_1\cdots N_g$.

For a vertex $\delta\in \Gamma$, let
$M_{\delta} := \sum_{a\in G} m_{(a\delta)} = \sum_{a\in G}(a^*\nu)(\varphi_{\delta})$.
The multiplicities ${M}_{\delta}$ are the same for the vertices from one
orbit. Therefore they depend only on the corresponding vertex in the quotient graph $\check \Gamma$.
All the multiplicities $M_\sigma$, $\sigma\in \check \Gamma$, are different and for $\sigma, \tau\in \Gamma$ one has
$M_\sigma=M_\tau$ if and only if $\tau = a \sigma$ for some $a\in G$.

The A'Campo type formula for the Poincar\'e series of a collection of plane divisorial valuations (see~(\ref{ACampo}))
implies that the series $\zeta_{\nu}(t)$ is represented by the rational function given by the equation
\begin{eqnarray}
 \zeta_{\nu}(t)&=&\prod_{q=0}^g \left(1-t^{M_{\sigma_q}}\right)^{-h_0/h_{j(q)}}\cdot
 \prod_{q=1}^g \left(1-t^{M_{\tau_q}}\right)^{h_0/h_{j(q)}}\cdot\nonumber\\
 &\cdot&\prod_{j=1}^{\ell} \left(1-t^{M_{\rho_j}}\right)^{(h_0/h_j)-(h_0/h_{j-1})}\cdot
 \left(1-t^{M_{\nu}}\right)^{-h_0/h_{\ell}},\label{zeta}
\end{eqnarray}
where $h_{j(q)}$ is the order of the isotropy subgroup of the vertex $\sigma_{q}$,
 $M_{\tau_q}=N_q M_{\sigma_q}$.

Any series in $t$ with integer coefficients and with the initial term $1$ can be in a unique way
written as the product $\prod\limits_{m\ge 1} \left(1-t^m\right)^{s(m)}$ with $s(m)\in \Z$ (in general
an infinite one). The same statement holds for series in several variables.
The equation~(\ref{zeta}) implies that the product in the right hand side of the
equation
\begin{equation}\label{zeta_decomp}
 \zeta_{\nu}(t)=\prod\limits_{m\ge 1} \left(1-t^m\right)^{s(m)}
\end{equation}
is finite. In~(\ref{zeta}) some exponents of $t$ in the binomials may coincide.
This happens if and only if the corresponding vertices coincide. Namely, $\rho_1$ may coincide
with $\sigma_0$, each $\rho_i$, $i\ge 1$, may coincide with a certain $\tau_q$, and, finally, $\nu$
may coincide with $\tau_g$. In the first and in the latter cases this can lead to the situation
when the corresponding binomial ``is not seen'' in the decomposition~(\ref{zeta_decomp}) because
the corresponding exponents cancel. Moreover, if $\rho_1=\sigma_0$, the corresponding factors
cancel if and only if $h_0/h_1=2$.

Let us take the minimal $m$ such that the binomial $(1-t^m)$ appears in the
decomposition~(\ref{zeta_decomp}) (i.\ e., $s(m)\ne 0$). If $s(m)=-1$, one has $m=M_{\sigma_0}$ (and
$\rho_1\neq \sigma_0$).
If $s(m)<-1$, then $s(m)=-2$, $\rho_1=\sigma_0$, $h_0/h_1=2$, and either $m=M_{\sigma_1}$, or
$g=0$, $\ell=1$ and $m=M_{\nu}$. (The latter option is a degenerate one and is
analogous to the one shown on the left hand side of Figure~\ref{fig1}.)
If $s(m)>0$, one has $\rho_1=\sigma_0$. There are two options.
Either $h_0/h_1=2$ and the binomial $\left(1-t^{M_{\sigma_0}}\right)$ does not appear in the
decomposition~(\ref{zeta_decomp}), or $h_0/h_1>2$, $m=M_{\sigma_0}$, $s(m)=(h_0/h_1)-2$.

If $\rho_1=\sigma_0$ and the binomial $\left(1-t^{M_{\sigma_0}}\right)$ does not appear in the
decomposition~(\ref{zeta_decomp}), the binomial of the form
$\left(1-t_1^{M_{\sigma_0}}t_2^{k_2}\cdots t_r^{k_r}\right)$ is present in the corresponding
decomposition of $\zeta(t_1, t_2, \ldots, t_r)$ (due to the conditions imposed on the
resolution graph before Theorem~\ref{theo1}: in this case there are more than two
edges
of the graph $\Gamma$ at the vertex $\sigma_0$) and the corresponding multi-exponent
$(M_{\sigma_0},k_2,\cdots, k_r)$ is the minimal one in the decomposition. Therefore
$M_{\sigma_0}$ can be determined in this case as well.

Now let us find all the exponents $M_{\sigma_i}$, $M_{\tau_i}$ for $i\ge 1$ (maybe except $M_{\tau_g}$),
$M_{\rho_j}$, $j=1,\ldots, \ell$, and the ratios $h_{j-1}/h_j$. Let us take all the binomials
$(1-t^m)$ with $m>{M_{\sigma_0}}$ in the decomposition~(\ref{zeta_decomp}) with
negative exponents $s(m)$.
All of them except possibly the biggest one are the multiplicities $M_{\sigma_i}$. (The biggest one
may coincide with $M_{\nu}$.)
Let $M_{\sigma_1}<M_{\sigma_2}<\ldots<M_{\sigma_p}$ be these exponents in the increasing order.
One has either $p=g$ (in this case $\nu=\tau_g$) or $p=g+1$ (and $\sigma_p=\nu$).
Moreover, one has
$s(M_{\sigma_p})=h_0/h_{\ell}$ (the number of different valuations in the orbit
$\{a^*\nu\}$ of the valuation $\nu$).

Let us take the exponents $m$ in (\ref{zeta_decomp}) such that $M_{\sigma_0}<m<M_{\sigma_1}$.
All these exponents correspond to the splitting points (up to $\rho_{j(1)}$). This gives the values
$M_{\rho_j}$ for these $j$. Moreover, one has $s(M_{\rho_j})=(h_0/h_j)-(h_0/h_{j-1})$. This gives all
the ratios $h_0/h_j$ for all $j\le j(1)$.

Assume that we have detected all $M_{\rho_j}$ ($j\le j(q)$) and $M_{\tau_i}$ ($i<q$) smaller
than $M_{\sigma_q}$ and all the ratios $h_0/h_j$ for $j\le j(q)$. Let us take all the exponents
$m$ in (\ref{zeta_decomp}) such that $M_{\sigma_q}<m<M_{\sigma_{q+1}}$ with non-zero (and thus positive)
$s(m)$. The smallest among them is $M_{\tau_q}$. The vertex $\tau_q$ can either be the splitting point
$\rho_{j(q)+1}$ or not. The vertex $\tau_q$ is the splitting point $\rho_{j(q)+1}$ if and only if
$s(M_{\tau_q})>-s(M_{\sigma_q})$. In this case $s(M_{\tau_q})=2(h_0/h_{j(q)})-(h_0/h_{j(q)+1})$.
This equation gives $h_0/h_{j(q)+1}$. All the remaining exponents $m$ inbetween $M_{\tau_q}$ and
$M_{\sigma_{q+1}}$ (with $s(m)$ positive) correspond to the splitting points $\rho_j$ (with $j$ up
to $j(q+1)$). As above the ratio $h_0/h_j$ is determined by
$s(M_{\rho_j})=(h_0/h_j)-(h_0/h_{j-1})$.

For $1\le q<p$ one has $M_{\tau_q}$ is a multiple of $M_{\sigma_q}$ and moreover
$M_{\tau_q}/M_{\sigma_q}=N_q$.
Now we compute the multiplicities $m_{\sigma_q}$ for $0\le q\le p$ and $m_{\rho_j}$ for $1\le j\le \ell$
(and finally determine $m_\nu$).

For $\sigma_q\le\rho_1$ one has
$M_{\sigma_q}=h_0 m_{\sigma_q}$ and $M_{\rho_1}=h_0 m_{\rho_1}$. These equations
give all the generators $m_{\sigma_q}$ of the semigroup of values of the valuation
with $\sigma_q\le\rho_1$ and also $m_{\rho_1}$.

For $j\ge 1$, let $\sigma_{q(j)}$ be the minimal dead end greater than $\rho_{j}$
(i.e. there are the dead ends $\sigma_{q(j)}$, \dots, $\sigma_{q(j+1)-1}$
inbetween
$\rho_{j}$ and $\rho_{j+1}$). Let us consider the dead ends $\sigma_q$ such that
$\rho_1<\sigma_q<\rho_2$. One has
$$
M_{\sigma_{q(1)}}= h_1 m_{\sigma_{q(1)}} + ( h_0 -
h_1)m_{\rho_1}
=
 h_1  m_{\sigma_{q(1)}} + (M_{\rho_{1}} -  h_{1}  m_{\rho_{1}})\,.
$$
For
$\rho_1 < \sigma_{q(1)} < \sigma_{q(1)+1} < \sigma_{q(1)+2} < \cdots <
\sigma_{q(2)-1} < \rho_2$, one has
\begin{eqnarray*}
M_{\sigma_{q(1)+1}}&=& h_1 m_{\sigma_{q(1)+1}} +
(M_{\rho_{1}} - h_{1} m_{\rho_{1}})N_{q(1)}\,,\\
M_{\sigma_{q(1)+2}}&=&
h_1 m_{\sigma_{q(1)+2}} +
(M_{\rho_{1}} -  h_{1} m_{\rho_{1}})N_{q(1)}N_{q(1)+1}\,,\\
{\ }&{\ }&\ldots\\
M_{\sigma_{q(2)-1}}&=&
h_1 m_{\sigma_{q(2)-1}} +
(M_{\rho_{1}} - h_{1}  m_{\rho_{1}})N_{q(1)}N_{q(1)+1}\cdot\ldots\cdot
N_{q(2)-2}\\
M_{\rho_2}&=&
h_1  m_{\rho_2} +
(M_{\rho_{1}} -  h_{1}  m_{\rho_{1}})N_{q(1)}N_{q(1)+1}\cdot\ldots\cdot
N_{q(2)-1}\,.
\end{eqnarray*}
These equations
give us the numbers $m_{\sigma_q}$ with $\sigma_q < \rho_2$ and also
$m_{\rho_2}$.

Assume that we have determined all the numbers $m_{\sigma_{q}}$ for
$q<q(j)$ and also the number
$m_{\rho_{j}}$. Let us consider the dead ends $\sigma_q$ such that
$\rho_{j}<\sigma_q<\rho_{j+1}$. One has
\begin{eqnarray*}
M_{\sigma_{q(j)}}&=&  h_{j}  m_{\sigma_{q(j)}} +
(M_{\rho_{j}} -   h_{j}  m_{\rho_{j}})\,,\\
M_{\sigma_{q(j)+1}}&=&  h_{j}  m_{\sigma_{q(j)+1}} +
(M_{\rho_{j}} -   h_{j}  m_{\rho_{j}})N_{q(j)}\,,\\
M_{\sigma_{q(j)+2}}&=&  h_{j}  m_{\sigma_{q(j)+2}} +
(M_{\rho_{j}} -   h_{j}  m_{\rho_{j}})N_{q(j)}N_{q(j)+1}\,,\\
{\ }&{\ }&\ldots\\
M_{\sigma_{q(j+1)-1}}&=&  h_{j}  m_{\sigma_{q(j+1)-1}} +
(M_{\rho_{j}} -   h_{j}  m_{\rho_{j}})N_{q(j)}N_{q(j)+1}\cdots N_{q(j+1)-2}\,,\\
M_{\rho_{j+1}}&=&
  h_{j}  m_{\rho_{j}} +
(M_{\rho_{j}} -   h_{j}  m_{\rho_{j}})N_{q(j)}N_{q(j)+1}\cdot\ldots\cdot N_{q(j+1)-1}\,.
\end{eqnarray*}
These equations
give all the numbers $m_{\sigma_{q}}$ with
$q<q(j+1)$ and also
$m_{\rho_{j+1}}$.

This procedure gives us the numbers $m_{\sigma_q}$ for all $q\le p$. If
$\gcd(m_{\sigma_0}, m_{\sigma_1}, \ldots, m_{\sigma_{p-1}}) = 1$ then $p=g+1$,
$\sigma_p=\nu$. Otherwise $p=g$, $\nu = \tau_g$. In this way one determines all the
numbers $m_{\sigma_q}$ for $0\le q \le g$ (i.e. the generators of the semigroup of
values of the valuation $\nu$);  $m_{\rho_j}$ and $h_j$ for $1\le j\le \ell$ and
$m_{\nu}$.
This gives us the minimal resolution graph of the set of valuations $\{a^*\nu | a\in
G\}$.

To determine the (minimal) resolution graph of the collection of valuations
$\{a^*\nu_i | i=1\ldots, r; a\in G\}$ one has to determine the separation point
$\delta_{ij}$ between each two valuations $\nu_i$ and $\nu_j$ in the graph
$\check \Gamma = \Gamma/G$. For convenience let us assume that $i=1$ and $j=2$. Let
\begin{equation}\label{eqntwo}
\zeta(t_1,t_2,1\ldots,1) = \prod (1- t_1^{M_1}t_2^{M_2})^{s(M_1,M_2)} ,
\end{equation}
$s(M_1,M_2)\in \Z$, be the decomposition of the series $\zeta(t_1,t_2,1,\ldots,1)$
into the product of binomials.
The separation point $\delta_{12}$ corresponds to the maximal exponent present in the
decomposition (\ref{eqntwo}) (i.~e., such that $s(M_1,M_2)\neq 0$) with
$$
\frac{M_2}{M_1} = \frac{M_{\sigma_0 2}}{M_{\sigma_0 1}} .
$$
(If there is no such $(M_1,M_2)$, one has $\delta_{12} = \sigma_0$.)
To reduce the situation to the standard statements about topology of curves one has
to determine the multiplicities $m_1$ and $m_2$ of the separation point in the graph
$\check \Gamma = \Gamma/G$. Let $\delta'$ (respectively $\delta''$) be the vertex of
the minimal resolution graph $\Gamma_1$ (respectively $\Gamma_2$) of the valuation
$\nu_1$ (respectively $\nu_2$) such that $M_{\delta'}=M_1$ in $\Gamma_1$
($M_{\delta''}=M_2$ in $\Gamma_2$). Such vertex $\delta'$ (respectively $\delta''$)
either does not exist or is a unique one. Moreover, either $\delta'$, or $\delta''$
(or both) are present in the graphs $\Gamma_1$ and $\Gamma_2$ respectively. If
exactly one of $\delta'$ and $\delta''$ exists, without loss of generality we can
assume that this is $\delta'$. If both of them exist, we can assume (again
without loss of generality) that the age of $\delta'$ is smaller or equal to the age
of $\delta''$. One can see that $\delta'$ is the separation point in the resolution
graph $\Gamma_{\{ 1 2\}}$ of the pair $\{\nu_1, \nu_2\}$. The corresponding
multiplicity $m_{\delta'}$ in $\Gamma_1/G$ can be found in the same way as
$m_{\sigma_i}$, $m_{\tau_i}$ and $m_{\rho_j}$ above. The multiplicity $m_{\delta''}$
of the separation point in the graph $\Gamma_{\{1 2\}}$ corresponding to the
valuation $\nu_2$ is determined from the equality
${m_{\delta''}}/{m_{\delta'}} = {M_2}/{M_1}$.
\end{proof}

\section{Topology of curve valuations}\label{sec3}

Here we shall discuss collections of curve valuations on $\OO_{\C^2,0}$. Let
$(C_i,0)\subset (\C^2,0)$, $i=1,\ldots, s$, be irreducible plane curve singularities
and let $\nu_i$ be the (curve) valuation on $\OO_{\C^2,0}$ defined by the branch
$C_i$. Let $\pi: (X,{\cal D})\to (\C^2,0)$ be an embedded resolution of the curve
$C= \bigcup\limits_{i=1}^s C_i$. The exceptional divisor ${\cal D}$ is the union of
its irreducible components $E_{\sigma}$. For $i=1,\ldots,s$, let $E_{\alpha_i}$ be the
component of ${\cal D}$ intersecting the strict transform of the branch $C_i$.
The dual graph $\Gamma$ of the resolution $\pi$ is the graph whose vertices
correspond to the components $E_{\sigma}$ of the exceptional divisor ${\cal D}$ and
to the components $C_i$ of the curve $C$, the latter ones are depicted by arrows.
Two vertices are connected by an edge if and only if the corresponding components
intersect.
Pay attention that several arrows may be connected with one and the same vertex.
Each vertex $\sigma$ (corresponding to the divisor $E_{\sigma}$) has its age.

As in Section~\ref{sec2},  let us assume that the graph $\Gamma$ carries an action of
a finite group $G$ preserving the ages of the vertices. In particular, this means
that the group acts on the set of arrows, that is on the components of $C_i$ of the
curve for $i=1,\ldots, s$. Assume that the components $\seq{C}1r$ ($r<s$) are
representatives of all the orbits of the $G$-action on the curve components. The
component obtained from $C_i$ by the $a$-shift ($a\in G$) will be denoted by $a C_i$
or by
$C_{a i}$ and the corresponding curve valuation will be denoted by $a^*\nu_i$ or
by $\nu_{a i}$.

Let $P(\{t_{a i}\})$ be the Poincar\'e series of the collection $\{\nu_{a i}\}$ of
the $r\abs{G}$ curve valuations (defined by the components of the curve $C$) and let
$$
\zeta (\seq t1r) := P(\seq t11, \seq t22, \ldots, \seq trr)
$$
with $\abs{G}$ identical variables in each group.
Pay attention that we do not
impose additional conditions like in the divisorial case in Section~\ref{sec2}.

\begin{theorem}\label{theo2}
The series $\zeta(\seq t1r)$ determines the topological
type of the curve $C = \bigcup\limits_{a\in G} \bigcup\limits_{i=1}^r C_{a i}$.
\end{theorem}

\begin{proof}
As in the proof of Theorem~\ref{theo1},  we have to show that the series
$\zeta(\seq t1r)$ determines the minimal resolution graph $\Gamma$ of the curve $C$.
The graph $\Gamma$ looks essentially like the resolution graph of the divisorial
valuations corresponding to the vertices
$a \alpha_i$ ($i=1,\ldots, r$, $a\in G$) (see  Figure~\ref{fig2} for $r=1$) with
several
(possibly one)
arrows attached to each vertex $a \alpha_i$. Let $\check{\Gamma} = \Gamma/G$ be the
quotient of the graph $\Gamma$ by the $G$-action (like in Figure~\ref{fig3} for
$r=1$); some
arrows have to be added. In general (if among the arrows attached to one vertex
$a \alpha_i$ one has different representatives  of one $G$-orbit), the graph $\check \Gamma$
is not the minimal resolution graph of a curve, but may be somewhat enlarged (see
Figure~\ref{fig4}; in the minimal resolution graph the arrow $C_i$ is attached to
the vertex $\tau$ and the ``tail" between $\tau$ and $\rho$ does not exist).

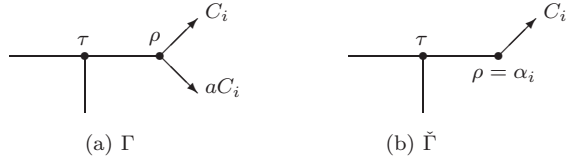
\begin{figure}[h]
$$
\unitlength=0.50mm
\begin{picture}(80.00,40.00)(50,10)
\thinlines

\put(20,30){\line(1,0){20}}
\put(40,30){\circle*{2}}
\put(38,33){\scriptsize{$\tau$}}
\put(40,30){\line(0,-1){15}}
\put(40,30){\line(1,0){20}}
\put(60,30){\circle*{2}}
\put(57,34){{\scriptsize$\rho$}}
\put(60,30){\vector(1,1){10}}
\put(60,30){\vector(1,-1){10}}
\put(72,40){{\scriptsize $C_i$}}
\put(72,20){{\scriptsize $aC_i$}}
\put(40,5){{\scriptsize (a) $\Gamma$}}

\put(110,30){\line(1,0){20}}
\put(130,30){\circle*{2}}
\put(128,33){\scriptsize{$\tau$}}
\put(130,30){\line(0,-1){15}}
\put(130,30){\line(1,0){20}}
\put(150,30){\circle*{2}}
\put(143,24){{\scriptsize$\rho=\alpha_i$}}
\put(150,30){\vector(1,1){10}}
\put(162,40){{\scriptsize $C_i$}}
\put(120,5){{\scriptsize (b) $\check \Gamma$}}

\end{picture}
$$
\caption{Parts of the graphs $\Gamma$ and $\check \Gamma = \Gamma/G$.
}
\label{fig4}
\end{figure}

As above we assume the graph $\check \Gamma$ to be embedded into the graph $\Gamma$.
(This assumption determines, in particular, the choice of representatives $C_i$, $i=1,\ldots, r$, from
the $G$-orbits by $C_j\in \check \Gamma$.)

Let $m_{\delta_i}$, $M_{\delta_i}$ and $\underline{M}_{\delta}$ be defined as in
Section~\ref{sec2}. Let $k_i$ be the order of the isotropy group of the branch $C_i$,
$i=1,\ldots,r$. One can see that, for
$i, j\in
\{1,\ldots
,r\}$,
$m_{\alpha_i j}$ is just the intersection multiplicity between the curves $C_i$ and
$C_j$ and
$$
M_{\alpha_i j} = \sum_{a\in G} m_{(a \alpha_i) j} = \sum_{a\in G} (a^*\nu_j)(h_{\alpha_i}) =
(C_i, \bigcup_{a\in G} a C_j) =
(C_j, \bigcup_{a\in G} a C_i) =
M_{\alpha_j i}\; .
$$

In the case of curve valuations the ``projection formula" is different from that for
divisorial valuations. Namely,
for $i_0\in \{1,\ldots, r\}$ one has:
\begin{equation}\label{projection}
\zeta_{\{\nu_i\}}(\tt)_{|_{t_{i_0}=1}} =
(1-\tt^{\MM_{\alpha_{i_0}}})^{\abs{G}/k_{i_0}}_{|_{t_{i_0}=1}} \cdot
\zeta_{\{\nu_i\}_{i\ne i_0}}(t_1,\ldots,t_{i_0 -1}, t_{i_0 +1}, \ldots,t_r)\,.
\end{equation}
(This can be easily deduced from the A'Campo type formula for the
Poincar\'e series of a collection of curve valuations.) Using
(\ref{projection}) several times for the indices $i\neq i_0$,
one gets:
\begin{equation}\label{projection2}
\zeta_{\{\nu_i\}}(\tt)_{|_{t_i=1, i\neq i_0}} =
\zeta_{\nu_{i_0}}(t_{i_0}) \cdot
\prod_{i\neq i_0} (1- t_{i_0}^{M_{\alpha_i i_0}})^{\abs{G}/k_i} \;.
\end{equation}
Using the fact that $M_{\alpha_i i_0}=M_{\alpha_{i_0} i}$, one sees that all the
multiplicities $M_{\alpha_{i} i_0}$, $i\neq i_0$,  are contained in $\underline{M}_{\alpha_{i_0}}$ as the components of it.

Using the induction on the number $r$ of the curve orbits,
equations (\ref{projection}) and (\ref{projection2}) mean that in order to describe
the (minimal) resolution graph $\Gamma$ of the curve $C$, we have:

\noindent 1) To describe the minimal resolution graph of the curve
$\bigcup\limits_{a\in G} a C_{i_0}$ (i.~e., for $r=1$) through the zeta function
$\zeta(t) = \zeta_{\nu_{i_0}}(t)$ (i.e to adapt the corresponding description
for one divisorial valuation in the proof of Theorem~\ref{theo1} to the curves case).

\noindent 2) To detect the binomial $(1-\tt^{\underline{M}_{\alpha_{i_0}}})$ and the number
$\abs{G}/k_{i_0}$ corresponding to (at least) one index $i_0\in \{1,\ldots, r\}$.
Using (\ref{projection}) this permits to get the series zeta for the remaining
$r-1$ valuations $\{\nu_i\}_{i\ne i_0}$. The induction assumes that the minimal resolution graph
of the collection $\{\nu_i\}_{i\ne i_0}$ is determined by the series $\zeta(\ldots)$. This gives, in
particular, all the numbers $\abs{G}/k_{i}$ for $i\ne i_0$. Using (\ref{projection2}) one gets
the series zeta and thus the minimal resolution graph for the valuation $\nu_{i_0}$.

\noindent 3) To determine the separation point of the curves $C_{i_0}$ and $C_i$
($i\neq i_0$) in the minimal resolution graph of these two curves.

Let us start with the point 1), i.~e.
let us assume that $r=1$ and $\zeta(t) = P_{\{a \nu_1\}}(t,\ldots,t)$ with
$\abs{G}$ identical variables in the right hand side, $P_{\{a \nu_1\}}(\seq
t1{\abs{G}})$ is
the Poincar\'e series of the set $\{a \nu_1\}$ of curve valuations (it is possible
that
$a \nu_1 = a' \nu_1$ for $a\neq a'$).
We have  to show how the minimal resolution graph of the curve
$\bigcup\limits_{a\in G} a C_1$ can be determined from the zeta function $\zeta(t)$.

Mostly this procedure repeats the one for the divisorial case. There is only one
essential difference. In the divisorial case we had to assume that the multiplicity
$M_{\sigma_0}$ can be determined at the first step of the consideration. This lead to
the exceptions which had to be excluded. The knowledge of $M_{\sigma_0}$ permitted us
to determine $m_{\sigma_i}, m_{\tau_i}$ for $i\geq 1$, $m_{\rho_j}$, {\dots}
In particular this permitted us to determine $m_{\tau_g}$ (even if the corresponding
binomial was absent in the decomposition of $\zeta(t)$).
However, if, by a chance, one knows $M_{\tau_g}$ (in the divisorial case this means
that $\tau_g\neq \nu$), one can recover $M_{\sigma_0}$ from the values $M_{\sigma_i}$
and $M_{\tau_i}$ for $i=1,\ldots, g$.
This happens in the curve case since the binomial
$(1-t^{M_{\tau_g}})$ in the decomposition of $\zeta(t)$ is always present with a
positive exponent.

The procedure described in the proof of Theorem~\ref{theo1} permits us to determine
all $M_{\sigma_i}$, $M_{\tau_i}$ for $i=1,\ldots , g$ (including $M_{\tau_g}$ !) and
also $M _{\rho_j}$ and $h_0/h_j$ for all $j$ except possibly the first one if
$\rho_1=\sigma_0$ and
$h_0/h_1 =2$. We have $M_{\tau_i} = N_i M_{\sigma_i}$, $i=1,\ldots,g$. It is known
that $m_{\sigma_0} = N_1\cdots N_g$ and therefore
$M_{\sigma_0} = h_0 m_{\sigma_0} = h_0 N_1 \cdots N_g$.
When $M_{\sigma_0}$ is known, all the multiplicities $m_{\sigma_i}$, $m_{\tau_i}$ and
$m_{\rho_j}$ can be determined in the same way as in Section~\ref{sec2}. This gives
the minimal resolution graph for one curve valuation.

Let us move to the point 2), i.~e., detect the binomial
$(1-\tt^{\underline{M}_{\alpha_{i_0}}})$ corresponding to an index
$i_0\in\{1,\ldots,r\}$ and the number $\abs{G}/k_{i_0}$.

Let us assume $r>1$ and
let us fix $j,k\in \{1,\ldots,r\}$. The separation point
$s(\alpha_j,\alpha_k)\in
\Gamma$
of $\alpha_j$ and $\alpha_k$ is
defined by the condition
$[\sigma_0,\alpha_j]\cap [\sigma_0,\alpha_k]=[{\bf 1},s(\alpha_j,\alpha_k)]$. Here $[\sigma_0,\sigma]$ is
the geodesic in the graph $\Gamma$ joining the first vertex $\sigma_0$ with the vertex
$\sigma$. Let us recall that we assume the graph $\check \Gamma$ to be embedded into the graph $\Gamma$ and $\alpha_i\in \check \Gamma$ for all $i$. This implies that
$ s(j,k):= s(\alpha_j,\alpha_k) \in \check \Gamma$ (and
$s(\alpha_j,\alpha_k) \ge s(\alpha_j, a \alpha_k)$ for $a\in G$).

The whole graph $\check \Gamma$ is constituted by its ``skeleton":
$SK = \bigcup\limits_{i=1}^{r}[\sigma_0, \alpha_i]$ and the segments connecting the
dead ends (i.~e., the vertices $\sigma$ such that $\chi(\overset{\circ}{E}_{\sigma})=1$)
with $SK$. The ratio $M_{\sigma j}/M_{\sigma k}$ is constant for $\sigma$ in $[\sigma_0,s(j,k)]$ and
is a strictly increasing function for $\sigma\in [s(i,j),\alpha_j]\subset \check \Gamma$.
Moreover, this ratio is also constant along the segments connecting the dead ends with $SK$.
The described properties of the ratios $M_{\sigma i}/M_{\tau i}$ imply
that, for each index $i\in \{1,\ldots, r\}$, one has:
\begin{eqnarray}
\frac{1}{M_{\alpha_i i}} \MM_{\alpha_i} & \le &
\frac{1}{M_{\delta i}} \MM_{\delta} \quad \forall \delta\in \check \Gamma\; ,
\nonumber \\
\frac{1}{M_{\alpha_i i}} \MM_{\alpha_i} & < &
\frac{1}{M_{\tau i}} \MM_{\tau} \quad \forall \tau \in SK;\  \tau \neq \alpha_i\; . \label{sk}
\end{eqnarray}
(Here we use the partial order
$\MM=(\seq M1r)\le \MM'=(\seq {M'}1r)$ if $M_i\le M'_i$ for all
$i=1,\ldots,r$;
$\MM < \MM'$ if $\MM\le \MM'$ and $M_i < M'_i$ for at least one $i$.)

Let $\sigma\in \check\Gamma$ be such that the multiplicity
$\MM_{\sigma} = (M_{\sigma 1}, \ldots, M_{\sigma r})$ is
a maximal one among the set of exponents $\MM_{\delta}$ appearing in the factorization
\begin{equation}\label{factor}
\zeta (\tt) = \prod_{\delta\in \check\Gamma\; , \; s(\underline{M}_{\delta})\neq 0}
(1-\tt^{\MM_\delta})^{s(\MM_{\delta})} \; ,
\end{equation}
i.~e., such that
$s(\underline{M}_{\delta})\neq
0$.
(An element $\MM$ from a subset of $\Z_{\ge 0}^{r}$ is maximal if there are no
elements greater than $\MM$.)
The maximality of $\MM_{\sigma}$ implies that, in the (minimal) resolution process, we do not
blow-up a point of the corresponding divisor $E_{\sigma}$. Therefore there exists an
index $i\in \{1,\ldots, r\}$ such that $\sigma=\alpha_j$. Let
$B(\sigma)$ be the set of such indices. Note that the exponent $s(\MM_{\sigma})$
of the binomial $(1-\tt^{\MM_{\sigma}})$ is
$s(\MM_{\sigma})= - n \chi(\overset{\circ}{E}_{\sigma}) $ for some positive integer
$n$ and therefore, if $\sigma=\alpha_j$ for some $j$, one has
$s(\MM_{\sigma})>0$.

Let $A(\sigma)\subset \{1,\ldots,r\}$ be the set of indices $i$, $1\le i \le r$,
such that
$$
\frac{1}{M_{\sigma i}} \MM_{\sigma} \le
\frac{1}{M_{\delta i}} \MM_{\delta}
$$
for all $\delta\in \check \Gamma$ with $s(\MM_{\delta})\neq 0$. The
equations (\ref{sk}) imply that $B(\sigma)\subset A(\sigma)$, however $A(\sigma)$
could contain some indices $\ell$ such that $\alpha_{\ell}\neq \sigma$.

Let us assume that there exists $\ell\in A(\sigma)$ such that
$\alpha_{\ell}\neq \sigma$. By the equations (\ref{sk}) this implies that
$\sigma\in [\sigma_0, \alpha_{\ell}]$ and for all $\delta\in [\sigma, \alpha_{\ell}]$,
$\delta\neq \sigma$, one has that $s(\MM_{\delta})=0$ and so
$\chi(\overset{\circ}{E}_{\delta})=0$. As a consequence the age of $\alpha_{\ell}$
is smaller than the one of $\sigma$ and $\alpha_{\ell}$ is  an end point
of the resolution graph of a branch $C_j$ with $j \in B(\sigma)$. In this case one
has that
$M_{\sigma \ell}<M_{\sigma j}$. On the other hand it is clear that
$M_{\sigma j} = M_{\sigma i}$ if $i,j \in B(\sigma)$ and so one can detect the
indices of $B(\sigma)$ as those $j\in A(\sigma)$ such that
$M_{\sigma j}\ge M_{\sigma \ell}$ for all $\ell\in A(\sigma)$.

For $i_0 \in B(\sigma)$, the binomial $(1-\tt^{\MM_{\sigma}})$ appears in (\ref{factor})
with the exponent
$s(\MM_{\sigma}) = -\chi(\overset{\circ}{E}_{\sigma})(\abs{G}/k_{i_0})$. Thus in
order to finish the proof one has to find an index $i_0\in B(\sigma)$ and to compute
$\chi(\overset{\circ}{E}_{\sigma})$.

First of all, if $(\sigma, \delta)$ is an edge at $\sigma$, the maximality of
$\MM_{\sigma}$ implies that the age of $\delta$ is smaller than the one of $\sigma$.
So, the number $\epsilon(\sigma)$ of edges at $\sigma$ is $\le 2$. The case
$\epsilon(\sigma)=0$ is only possible in the trivial case of two smooth and transversal
branches, in this case the Poincar\'e series (and so $\zeta(\tt)$) is equal to $1$
and we can omit this situation.

Let us assume that $A(\sigma)\neq B(\sigma)$, then
$\epsilon(\sigma) \ge \# (A(\sigma)\setminus B(\sigma))$ and so if
$\# (A(\sigma)\setminus B(\sigma))=2$ (see Figure~\ref{fign}(a)) then
$-\chi(\overset{\circ}{E}_{\sigma})=\# B(\sigma)$.
The case
$\# (A(\sigma)\setminus B(\sigma))= \epsilon(\sigma) = 1$ is possible only
if $A(\sigma)=\{1,\ldots, r\}$, all the branches are smooth, all the branches
of
$B(\sigma)$ split
at the same vertex $\sigma$ and the one in $A(\sigma)\setminus B(\sigma)$ is
transversal to the others
(see Figure~\ref{fign}(b)).
This case is characterized by the fact that
$\zeta(\tt)= (1-\tt^{\MM_{\sigma}})^{(r-2)\abs{G}}$, $M_{\sigma j}=M_{\sigma j'}$ if
$j, j'\in B(\sigma)$; $M_{\sigma_{\ell}}=1$ for $\ell\notin B(\sigma)$.
In the remaining case one has
$\# (A(\sigma)\setminus B(\sigma))= 1$, $\epsilon(\sigma)=2$ and so
$-\chi(\overset{\circ}{E}_{\sigma}) = \# B(\sigma)$.
(See Figure~\ref{fign}(c)).

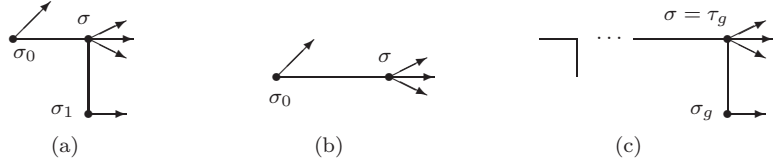
\begin{figure}[h]
$$
\unitlength=0.50mm
\begin{picture}(80.00,40.00)(10,0)
\thinlines
\put(-40,25){\scriptsize{$\sigma_0$}}
\put(-40,30){\line(1,0){20}}
\put(-40,30){\vector(1,1){10}}
\put(-40,30){\circle*{2}}
\put(-20,30){\circle*{2}}
\put(-20,30){\line(0,-1){20}}
\put(-20,10){\circle*{2}}
\put(-30,10){\scriptsize $\sigma_1$}
\put(-20,10){\vector(1,0){10}}
\put(-23,34){{\scriptsize$\sigma$}}
\put(-20,30){\vector(2,1){10}}
\put(-20,30){\vector(2,-1){10}}
\put(-20,30){\vector(1,0){12}}
\put(-30,0){{\scriptsize (a)}}

\put(30,20){\line(1,0){30}}
\put(30,20){\vector(1,1){10}}
\put(30,20){\circle*{2}}
\put(28,13){\scriptsize{$\sigma_0$}}
\put(60,20){\circle*{2}}
\put(57,24){{\scriptsize$\sigma$}}
\put(60,20){\vector(2,1){10}}
\put(60,20){\vector(1,0){12}}
\put(60,20){\vector(2,-1){10}}
\put(40,0){{\scriptsize (b)}}

\put(100,30){\line(1,0){10}}
\put(110,30){\line(0,-1){10}}
\put(115,30){\scriptsize{\ldots}}
\put(125,30){\line(1,0){25}}
\put(150,30){\circle*{2}}
\put(150,30){\line(0,-1){20}}
\put(150,10){\circle*{2}}
\put(140,10){\scriptsize{$\sigma_g$}}
\put(150,10){\vector(1,0){10}}
\put(133,35){{\scriptsize$\sigma=\tau_g$}}
\put(150,30){\vector(2,1){10}}
\put(150,30){\vector(2,-1){10}}
\put(150,30){\vector(1,0){12}}
\put(120,0){{\scriptsize (c)}}

\end{picture}
$$
\caption{The cases with $A(\sigma)\neq B(\sigma$).
}
\label{fign}
\end{figure}

As a consequence of the above discussion
we can assume that $A(\sigma)=B(\sigma)$.
First, assume that there exists $\delta\in \check \Gamma$ such that
$s(\MM_{\delta}) <0 $
and $\MM_{\sigma}= n \MM_{\delta}$ for some positive integer $n$.
The case in which the first vertex $\sigma_0 = \delta$ is the only one with
this condition is only reached if $A(\sigma)= \{1,\ldots,r\}$ and all the branches
are smooth and split at $\sigma$ (see Figure~\ref{fignn}(b)). This case is
characterized by the fact that
$\zeta(\tt) = (1-\tt^{\MM_\sigma})^{(r-1)\abs{G}}(1-\tt^{\MM\delta})^{-\abs{G}}$,
$\MM_{\delta} = (1,\ldots, 1)$ and
$\MM_{\delta} = (k,\ldots, k)$. Otherwise one has
$-\chi(\overset{\circ}{E}_{\sigma}) = \# A(\sigma)$ (see Figures~\ref{fignn}(a) and (c)).

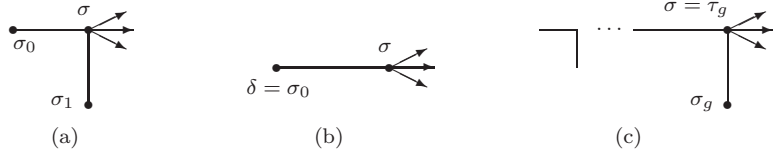
\begin{figure}[h]
$$
\unitlength=0.50mm
\begin{picture}(80.00,40.00)(10,0)
\thinlines
\put(-40,25){\scriptsize{$\sigma_0$}}
\put(-40,30){\line(1,0){20}}
\put(-40,30){\circle*{2}}
\put(-20,30){\circle*{2}}
\put(-20,30){\line(0,-1){20}}
\put(-20,10){\circle*{2}}
\put(-30,10){\scriptsize $\sigma_1$}
\put(-23,34){{\scriptsize$\sigma$}}
\put(-20,30){\vector(2,1){10}}
\put(-20,30){\vector(2,-1){10}}
\put(-20,30){\vector(1,0){12}}
\put(-30,0){{\scriptsize (a)}}

\put(30,20){\line(1,0){30}}
\put(30,20){\circle*{2}}
\put(22,13){\scriptsize{$\delta=\sigma_0$}}
\put(60,20){\circle*{2}}
\put(57,24){{\scriptsize$\sigma$}}
\put(60,20){\vector(2,1){10}}
\put(60,20){\vector(1,0){12}}
\put(60,20){\vector(2,-1){10}}
\put(40,0){{\scriptsize (b)}}

\put(100,30){\line(1,0){10}}
\put(110,30){\line(0,-1){10}}
\put(115,30){\scriptsize{\ldots}}
\put(125,30){\line(1,0){25}}
\put(150,30){\circle*{2}}
\put(150,30){\line(0,-1){20}}
\put(150,10){\circle*{2}}
\put(140,10){\scriptsize{$\sigma_g$}}
\put(133,35){{\scriptsize$\sigma=\tau_g$}}
\put(150,30){\vector(2,1){10}}
\put(150,30){\vector(2,-1){10}}
\put(150,30){\vector(1,0){12}}
\put(120,0){{\scriptsize (c)}}
\end{picture}
$$
\caption{The cases with $\MM_{\sigma}= n\MM_{\delta}$.
}
\label{fignn}
\end{figure}

Let us assume that
there exists $\tau\in \check \Gamma$ with
$s(\MM_{\tau})> 0$ such  that the difference
$\MM_{\sigma i} - \MM_{\tau i}$ is equal to $0$ for all the coordinates $i\notin A(\sigma)$ and is
equal to one and the same constant for those $i\in A(\sigma)$.
In this case
the number of edges at $\sigma$, $\epsilon(\sigma)$,
is equal to 1 and so we finish because
$-\chi(\overset{\circ}{E}_{\sigma}) = \# A(\sigma)-1$.

The only situation when $\epsilon(\sigma)=1$ and the above mentioned element $\tau$
does
not exists is the following one: all the branches from $A(\sigma)$ are
smooth and
split at the vertex $\sigma$; moreover there is no vertices $\delta\neq \sigma$ on the geodesic
$[\sigma_0, \sigma]$ with $s(\MM_{\delta})\neq 0$. Thus, in particular,
$\chi(\overset{\circ}{E}_{\sigma_0})=0$ and so there are two edges on
$\Gamma$ starting at $\sigma_0$. The case when this edges are conjugate by the action of the
group
$G$ is characterized by the fact that $\zeta(\tt) =
(1-\tt^{\MM_{\sigma}})^{(r-1)\abs{G}/2}$.
Otherwise both edges are invariant by the
action of the group, one of them corresponds to the indices from $A(\sigma)$ and the
other one to the remaining ones.
In this case we proceed as follows: let us consider
an element $\MM_{\sigma'}$, maximal among those $\MM_{\delta}$ with
$s(\MM_{\delta})\neq 0$ different from $\MM_{\sigma}$. For this new element
we reproduce
the steps we made before for $\sigma$.  Note that all the indices involved in this
new process are in the  complement $A'(\sigma)$ of $A(\sigma)$ so, there is no
conflict with the previous ones. If, by the previous methods, we are able to determine
an index $i_0\in A'(\sigma)$ and the Euler characteristic
$\chi(\overset{\circ}{E}_{\alpha_{i_0}})$ then we finish. Otherwise we have just a
similar situation for the vertex $\sigma'$, that is we have $\epsilon(\sigma')=1$,
the branches of $A(\sigma')$ are all the branches of $A'(\sigma)$; all of them are
smooth and split at $\sigma'$ and, moreover, there is no vertex $\delta$ on the
geodesic $[1,\sigma']$ such that $s(\MM_{\delta})\neq 0$ (see Figure~\ref{fignnn}).
This case is characterized by the fact that the function zeta is
$$
\zeta(\tt) = (1-\tt^{\MM_{\sigma}})^{(\#A(\sigma)-1)\abs{G}}
(1-\tt^{\MM_{\sigma'}})^{(\#A(\sigma')-1)\abs{G}}
$$
with $\MM_{\sigma i}=1$ (respectively
$\MM_{\sigma' i}=1$) if $i\notin A(\sigma)$ (respectively $i\notin A(\sigma')$) and
with $\MM_{\sigma i}=k$ (respectively
$\MM_{\sigma' i}=k'$) if $i\in A(\sigma)$ (respectively $i\in A(\sigma')$).

\begin{figure}[h]
$$
\unitlength=0.50mm
\begin{picture}(60.00,60.00)(-20,0)
\thinlines
\put(-30,30){\circle*{2}}
\put(-35,23){{\scriptsize$\sigma_0$}}
\put(-30,30){\line(3,1){42}}
\put(-30,30){\line(3,-1){63}}
\put(12,44){\circle*{2}}
\put(12,44){\vector(2,1){10}}
\put(12,44){\vector(2,-1){10}}
\put(30,40){{\scriptsize$A'(\sigma) = A(\sigma')$}}
\put(8,35){{\scriptsize$\sigma'$}}

\put(33,9){\circle*{2}}
\put(33,9){\vector(2,1){10}}
\put(33,9){\vector(2,-1){10}}
\put(33,9){\vector(1,0){12}}
\put(50,7){{\scriptsize $A(\sigma)$}}
\put(30,15){{\scriptsize$\sigma$}}
\end{picture}
$$
\caption{The degenerate case with $A(\sigma')=A'(\sigma)$ and
$\epsilon(\sigma)=\epsilon(\sigma')=1$.}
\label{fignnn}
\end{figure}
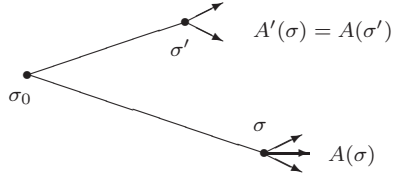

If the above procedure did not produce an index $i_0\in
\{1,\ldots, r\}$ and its Euler characteristic, then for the index $\sigma$ one has that
$\epsilon(\sigma)=2$ and so
$-\chi(\overset{\circ}{E}_{\sigma}) =\# A(\sigma)$. This finishes the determination
of an index $i_0$ from $\{1,\ldots, r\}$, of the Euler characteristic of
${E}_{\alpha_{i_0}}$, and thus of the exponent $\abs{G}/k_{i_0}$.

The method to determine the separation points between the curve $C_{i_0}$ and the
curves $C_i$ with $i\neq i_0$ is literally the same as in the divisorial case.

\end{proof}

\section{The order of the group and the function zeta}\label{sec4}

In Sections~\ref{sec2} and~\ref{sec3} we assumed that the order of the group $G$ was known.
This requirement is really needed in some cases. In Theorem~\ref{theo1} (under the described conditions)
the knowledge of the order of the group is not necessary. One does not use the order of the group
to find the multiplicities $M_{\sigma_i}$, $M_{\tau_i}$ and $M_{\rho_j}$ (for $r=1$). The same procedure
does not give the orders $h_j$ of the isotropy subgroups, but gives all the ratios $h_0/h_j$.
One should modify the equations
for the multiplicities $m_{\sigma_i}$ and $m_{\rho_j}$ used in the proof of Theorem~\ref{theo1} so
that they become equations with respect to $h_{\ell}\cdot m_{\sigma_i}$ for $i=0,
\ldots, p$ and
$h_{\ell}\cdot m_{\rho_j}$ for $j=1,\ldots, \ell$. The integer $h_{\ell}$ (the order
of the isotropy subgroup of the valuation $\nu$) is nothing else but
$\gcd{(h_{\ell}m_{\sigma_0},h_{\ell}m_{\sigma_1},\ldots, h_{\ell}m_{\sigma_p})}$.
The knowledge of $h_{\ell}$ and of the ratio $h_0/h_{\ell}$ gives the order $h_0$ of the group $G$.

In the setting of Theorem~\ref{theo2}, in some situations the multiplicity $M_{\sigma_0}$ can be
determined just at the very beginning. Namely, this is possible if, in the resolution graph of
the curve $C=\bigcup\limits_{a\in G}\bigcup\limits_{i=1}^{\infty}aC_i$, either the number of
edges at the (initial) vertex $\sigma_0$ is different from $2$, or it is equal to
$2$, but these
two edges are not interchanged by the group action. (This means that either the number of lines in the
tangent cone of the curve $C$ is different from $2$, or the tangent cone consists of
two lines not from the same
$G$-orbit.) If the multiplicity $M_{\sigma_0}$ is known, the way to determine the multiplicities
$m_{\sigma_i}$, $m_{\tau_i}$ and $m_{\rho_j}$ and therefore the resolution graph is the same as
in the divisorial case above.

The following example shows that, if the order of the group $G$ is not assumed to be known
and the multiplicity $M_{\sigma_0}$ cannot be determined in the described way, the topological
type of the curve singularity (in fact already with $r=1$) is not determined by the series
$\zeta(t)$. Let $C'$ be the (non-reduced) curve defined by the equation
$(y^2-x^3)^7(x^2-y^3)^7=0$ with the natural (non-trivial) action of the group of order $14$ on its components
and let $C''$ be the curve defined by the equation $(y^2-x^5)^5(x^2-y^5)^5=0$ with the natural
(non-trivial) action of the group of order $10$ on its components. The (minimal) resolution graphs of the
curves $C'$ and $C''$ are shown in Figure~\ref{fig5}.

\begin{figure}[h]
$$
\unitlength=0.50mm
\begin{picture}(150.00,70.00)(0,0)
\thinlines

\put(-30,30){\line(2,1){30}}
\put(-30,30){\line(2,-1){30}}
\put(-30,30){\circle*{2}}
\put(-40,28){\scriptsize{$4$}}

\put(0,45){\circle*{2}}
\put(3,48){\scriptsize{$10$}}
\put(0,45){\vector(1,0){20}}
\put(25,42){\scriptsize{$(7)$}}
\put(0,45){\line(-1,2){10}}
\put(-10,65){\circle*{2}}
\put(-5,62){\scriptsize{$5$}}

\put(0,15){\circle*{2}}
\put(3,18){\scriptsize{$10$}}
\put(0,15){\vector(1,0){20}}
\put(25,12){\scriptsize{$(7)$}}
\put(0,15){\line(-1,-2){10}}
\put(-10,-5){\circle*{2}}
\put(-5,-7){\scriptsize{$5$}}

\put(70,30){\line(4,1){60}}
\put(70,30){\line(4,-1){60}}
\put(70,30){\circle*{2}}
\put(60,28){\scriptsize{$4$}}

\put(100,37.5){\circle*{2}}
\put(98,40){\scriptsize{$6$}}
\put(130,45){\circle*{2}}
\put(133,48){\scriptsize{$14$}}
\put(130,45){\vector(1,0){20}}
\put(155,42){\scriptsize{$(5)$}}
\put(130,45){\line(-1,4){5}}
\put(125,65){\circle*{2}}
\put(130,63){\scriptsize{$7$}}

\put(100,22.5){\circle*{2}}
\put(98,15){\scriptsize{$6$}}
\put(130,15){\circle*{2}}
\put(133,18){\scriptsize{$14$}}
\put(130,15){\vector(1,0){20}}
\put(155,12){\scriptsize{$(5)$}}
\put(130,15){\line(-1,-4){5}}
\put(125,-5){\circle*{2}}
\put(130,-8){\scriptsize{$7$}}

\end{picture}
$$
\caption{The resolution graphs of the curves $C'$ and $C''$}
\label{fig5}
\end{figure}
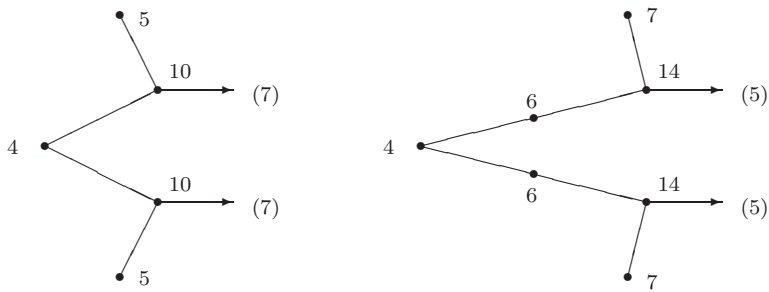

The numbers at the vertices (including the arrows) are the multiplicities of the corresponding
components in the zero divisors of the liftings of the corresponding functions (the left hand sides
of the equations) to the surfaces of resolution.
(These multiplicities define the ages of the vertices and thus the combinatorics
 of the resolutions in an obvious way.)
The A'Campo formula gives
$\zeta(t)=(1-t^{35})^{-2}(1-t^{70})^{2}$ in the both cases.

In the setting of Theorem~\ref{theo1} (i.~e., for divisorial valuations) the
possibility
to find the multiplicity $M_{\tau_g}$ in some cases (namely when the binomial $(1-t^{M_{\tau_g}})$
is present in the decomposition~(\ref{zeta_decomp})) does not permit, in general, to restore the
resolution graph and/or the order of the group. This is shown by the following example.
Let us consider two modification graphs shown on Figure~\ref{fig6} with the
divisorial
valuations corresponding to the vertices marked by the circles with the groups of orders
$14$ and $10$ in the left and in the right hand sides respectively (exchanging the two valuations
in each case).

\begin{figure}[h]
$$
\unitlength=0.50mm
\begin{picture}(150.00,160.00)(0,0)
\thinlines

\put(-10,120){\line(2,1){30}}
\put(-10,120){\line(2,-1){30}}
\put(-10,120){\circle*{2}}
\put(-20,118){\scriptsize{$4$}}

\put(20,135){\circle*{2}}
\put(23,138){\scriptsize{$10$}}
\put(20,135){\line(1,0){100}}
\put(127,135){\scriptsize{$(7)$}}
\put(20,135){\line(-1,2){10}}
\put(10,155){\circle*{2}}
\put(15,152){\scriptsize{$5$}}

\put(40,135){\circle*{2}}
\put(60,135){\circle*{2}}
\put(80,135){\circle*{2}}
\put(100,135){\circle*{2}}
\put(120,135){\circle*{2}}
\put(120,135){\circle{4}}

\put(38,138){\scriptsize{$11$}}
\put(58,138){\scriptsize{$12$}}
\put(78,138){\scriptsize{$13$}}
\put(98,138){\scriptsize{$14$}}
\put(118,140){\scriptsize{$15$}}

\put(20,105){\circle*{2}}
\put(23,108){\scriptsize{$10$}}
\put(20,105){\line(1,0){100}}
\put(127,105){\scriptsize{$(7)$}}
\put(20,105){\line(-1,-2){10}}
\put(10,85){\circle*{2}}
\put(15,87){\scriptsize{$5$}}

\put(40,105){\circle*{2}}
\put(60,105){\circle*{2}}
\put(80,105){\circle*{2}}
\put(100,105){\circle*{2}}
\put(120,105){\circle*{2}}
\put(120,105){\circle{4}}

\put(-30,30){\line(4,1){60}}
\put(-30,30){\line(4,-1){60}}
\put(-30,30){\circle*{2}}
\put(-40,28){\scriptsize{$4$}}

\put(0,37.5){\circle*{2}}
\put(-2,40){\scriptsize{$6$}}
\put(30,45){\circle*{2}}
\put(33,48){\scriptsize{$14$}}
\put(30,45){\line(1,0){140}}
\put(177,45){\scriptsize{$(5)$}}
\put(30,45){\line(-1,4){5}}
\put(25,65){\circle*{2}}
\put(30,63){\scriptsize{$7$}}

\put(50,45){\circle*{2}}
\put(70,45){\circle*{2}}
\put(90,45){\circle*{2}}
\put(110,45){\circle*{2}}
\put(130,45){\circle*{2}}
\put(150,45){\circle*{2}}
\put(170,45){\circle*{2}}
\put(170,45){\circle{4}}

\put(48,48){\scriptsize{$15$}}
\put(68,48){\scriptsize{$16$}}
\put(88,48){\scriptsize{$17$}}
\put(108,48){\scriptsize{$18$}}
\put(128,48){\scriptsize{$19$}}
\put(148,48){\scriptsize{$20$}}
\put(168,50){\scriptsize{$21$}}

\put(0,22.5){\circle*{2}}
\put(-2,15){\scriptsize{$6$}}
\put(30,15){\circle*{2}}
\put(33,18){\scriptsize{$14$}}
\put(30,15){\line(1,0){140}}
\put(177,14){\scriptsize{$(5)$}}
\put(30,15){\line(-1,-4){5}}
\put(25,-5){\circle*{2}}
\put(30,-8){\scriptsize{$7$}}

\put(50,15){\circle*{2}}
\put(70,15){\circle*{2}}
\put(90,15){\circle*{2}}
\put(110,15){\circle*{2}}
\put(130,15){\circle*{2}}
\put(150,15){\circle*{2}}
\put(170,15){\circle*{2}}
\put(170,15){\circle{4}}

\end{picture}
$$
\caption{The modification graphs definning the divisorial valuations.}
\label{fig6}
\end{figure}

The A'Campo type formula gives
$$
\zeta(t)=(1-t^{35})^{-2}(1-t^{70})^{2}(1-t^{105})^{-2}
$$
in the both cases.

\section{Unknotted links with the same Alexander polynomials}\label{sec5}

In the setting of Theorem~\ref{theo2} with $r=1$, all the components $a C_1$ of the
curve
$C=\bigcup\limits_{a\in G}aC_1$ are equisingular, that is have the same topological
type.
The attempt to understand whether for the statement to hold it is really necessary that
the components of the curve $C$ are obtained from one of them by a group action on a
modification graph or it is sufficient that all the components are equisingular led
to the following example.

Let $f'(x,y):=(x^3+y^{12})(y+x^2)(y+x^2)(y^3+x^{12})$, $f''(x,y):=(x^3+y^{15})(y+x^2)(y^4+x^{12})$.
The function germs $f'$ and $f''$ are products of function germs right equivalent to the function $x$.
Therefore all the components of the curves $C'=\{f'=0\}$ and $C''=\{f''=0\}$ are equisingular, moreover
the components of the curve $C'$ and the components of the curve $C''$ have the same topological type
(all of them are smooth). The algebraic links $L'=C'\cap S^3_{\eps}$ and $L''=C''\cap S^3_{\eps}$
consist of unknotted components (of $8$ of them each).

 \begin{figure}[h]
$$
\unitlength=0.50mm
\begin{picture}(120.00,60.00)(0,0)
\thinlines
\put(-30,30){\line(3,1){63}}
\put(-30,30){\line(3,-1){63}}
\put(33,51){\vector(1,0){21}}
\put(33,51){\vector(3,1){20}}
\put(33,51){\vector(3,-1){20}}
\put(33,9){\vector(1,0){21}}
\put(33,9){\vector(3,1){20}}
\put(33,9){\vector(3,-1){20}}
\put(-9,23){\vector(3,1){20}}
\put(12,16){\vector(3,1){20}}
\put(-30,30){\circle*{2}}
\put(-37,28){{\scriptsize$8$}}
\put(-9,37){\circle*{2}}
\put(-12,41){{\scriptsize$11$}}
\put(12,44){\circle*{2}}
\put(9,48){{\scriptsize$14$}}
\put(33,51){\circle*{2}}
\put(30,55){{\scriptsize$17$}}
\put(-9,23){\circle*{2}}
\put(-11,16){{\scriptsize$13$}}
\put(12,16){\circle*{2}}
\put(10,8){{\scriptsize$17$}}
\put(33,9){\circle*{2}}
\put(31,2){{\scriptsize$20$}}

\put(80,30){\line(3,1){84}}
\put(80,30){\line(3,-1){42}}
\put(164,58){\vector(1,0){21}}
\put(164,58){\vector(3,1){20}}
\put(164,58){\vector(3,-1){20}}
\put(122,16){\vector(3,1){20}}
\put(122,16){\vector(3,-1){20}}
\put(122,16){\vector(1,0){21}}
\put(122,16){\vector(4,-3){17}}
\put(101,23){\vector(3,1){20}}
\put(101,23){\circle*2{2}}
\put(98,15){{\scriptsize$13$}}
\put(122,16){\circle*{2}}
\put(118,6){{\scriptsize$17$}}

\put(80,30){\circle*{2}}
\put(78,23){{\scriptsize$8$}}
\put(101,37){\circle*{2}}
\put(99,30){{\scriptsize$11$}}
\put(122,44){\circle*{2}}
\put(120,37){{\scriptsize$14$}}
\put(143,51){\circle*{2}}
\put(145,44){{\scriptsize$17$}}
\put(164,58){\circle*{2}}
\put(162,51){{\scriptsize$20$}}
\end{picture}
$$
\caption{The resolution graphs of the functions $f'$ and $f''$.}
\label{fig7}
\end{figure}
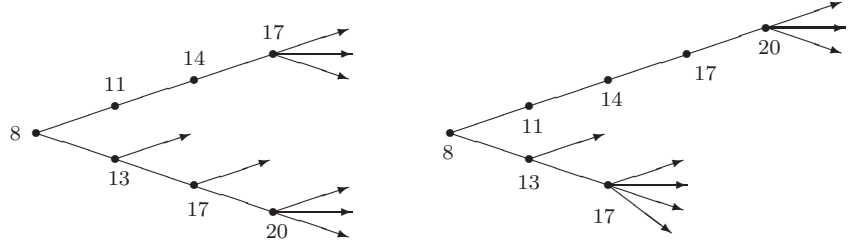

\begin{proposition}
 The curve singularities $C'$ and $C''$ (and thus the links $L'$ and $L''$) are topologically
 different. The monodromy zeta functions of the germs $f'$ and $f''$ (and thus the Alexander
 polynomials in one variable of the links $L'$ and $L''$) coincide.
\end{proposition}

\begin{proof}
 The resolution graphs of the function germs $f'$ and $f''$ are shown in Fig~\ref{fig7}.
 The numbers at the vertices are the multiplicities of the corresponding components in the
 zero divisor of the liftings of the functions $f'$ and $f''$ to the corresponding surfaces
 of resolution.
The graphs are topologically different and therefore the
 curve germs $C'$ and $C''$ (and the links $L'$ and $L''$) are topologically different as well.
 (In particular, these links have different Alexander polynomials in several variables.)
 The A'Campo formula gives that in the both cases
 $$
 \zeta(t)=(1-t^{13})(1-t^{17})^3(1-t^{20})^2.
 $$
\end{proof}

Addresses:

A. Campillo and F. Delgado:
IMUVA (Instituto de Investigaci\'on en
Matem\'aticas), Universidad de Valladolid.
Paseo de Bel\'en, 7. 47011 Valladolid, Spain.
\newline E-mail: campillo\symbol{'100}agt.uva.es, fdelgado\symbol{'100}agt.uva.es

S.M. Gusein-Zade:
Moscow State University, Faculty of Mathematics and Mechanics, Moscow, GSP-1, 119991, Russia.
\newline E-mail: sabir\symbol{'100}mccme.ru

\end{document}